\documentclass{amsart}

\usepackage{amscd,amssymb,graphics}
\usepackage{hyperref}
\usepackage{amsfonts}
\usepackage{amsmath}
\usepackage{enumerate}
\usepackage{tikz}

\input xy
\xyoption{all}

\usepackage{vmargin}
\setmarginsrb{20mm}{10mm}{20mm}{13mm}%
          {10mm}{7mm}{10mm}{10mm}

\newtheorem{thm}{Theorem}[section]

\newtheorem{corol}[thm]{Corollary}
\newtheorem{lemma}[thm]{Lemma}
\newtheorem{prop}[thm]{Proposition}

\newtheorem{definition}[thm]{Definition}
\numberwithin{equation}{section}
\theoremstyle{remark}
\newtheorem{remark}[thm]{Remark}
\newtheorem{example}[thm]{Example}

\newcommand{\ben}{\begin{enumerate}}
\newcommand{\een}{\end{enumerate}}

\def\R {{\mathbb R}}

\def\N{{\mathbb N}}

\def\Z {{\mathbb Z}}

\def\End{\operatorname{End}}
\def\Aut{\operatorname{Aut}}

\def\AT{\operatorname{AT}}

\def\Alt{\operatorname{Alt}}
\begin{document}	
	\title[The Addition theorem for two-step nilpotent torsion groups]{The Addition theorem for two-step nilpotent torsion groups}
	\author[]{Menachem Shlossberg}
	\address[M. Shlossberg]
	{\hfill\break School of Computer Science 
		\hfill\break Reichman University 
		\hfill\break 8 Ha'universita Street, Herzliya,  4610101
		\hfill\break Israel}
	\email{menachem.shlossberg@post.idc.ac.il}


\begin{abstract}
The Addition Theorem for the algebraic entropy of group endomorphisms  of torsion abelian groups  was proved in \cite{DGSZ}. Later, this result was extended to all abelian groups \cite{DGBabelian} and, recently, to all torsion finitely quasihamiltonian groups \cite{GBS}. In contrast,  when it comes to metabelian groups, 
 the 
 additivity of the algebraic entropy fails \cite{GBSp}.
Continuing the research within the class of locally finite groups,  we prove that the Addition Theorem holds for two-step nilpotent torsion groups. 
\end{abstract}
	\maketitle
	\section{Introduction}
	The algebraic entropy was first considered in  \cite{AKM} for endomorphisms of (discrete) abelian groups. 
	 Weiss \cite{W} and Peters \cite{Pet}  connected  this  entropy with the topological entropy  via  Bridge Theorems. We refer the reader     to \cite{DGSZ} for an extensive study of the algebraic entropy for endomorphisms of torsion abelian groups.

	Following 	\cite{DG-islam} we  now give the general definition for the algebraic entropy of endomorphisms of  (not necessarily abelian) groups.
		Let $G$ be a group and $\phi\in \End(G)$. For a finite subset $X$ of $G$ and $n\in \N_{+}$,  the $n$-th $\phi$-trajectory of $X$ is
	\[T_n(\phi,X)=X\cdot \phi(X)\ldots \cdot \phi^{n-1}(X). \] The algebraic entropy of $\phi$ with respect
	to $X$ is 
	$$H(\phi, X)=\lim_{n\to \infty}\frac{\ell(T_n(\phi,X))}{n},$$ 
	 where $\ell(T_n(\phi,X))=\log |T_n(\phi,X)|.$  The  algebraic entropy of $\phi$ is
	$h(\phi)=\sup\{H(\phi, X)|\ X\in \mathcal P_{fin}(G)\}$, where $\mathcal P_{fin}(G)$ is the family of all finite subsets of $X$.
	It is easy to see that $h(\phi)=\sup\{H(\phi, X)|\ X\in \mathcal C\}$, 	where $\mathcal C$ is any cofinal subfamily of  $\mathcal P_{fin}(G)$. 
	In particular, if $G$ is locally finite, then $\mathcal C$ can be chosen to be  $\mathcal F(G)$ the family of all finite subgroups of $G.$
	\vskip 0.3cm
	In this paper, we focus on the following property that is known as the Addition Theorem.
	\begin{definition}
Let $\mathfrak X$ be a class of  groups closed under taking subgroups and quotients.\ben \item	We say that $\AT(G,\phi,H)$ holds for a group $G\in \mathfrak X$, $\phi\in \End(G)$ and a $\phi$-invariant normal subgroup $H$ of $G$ if
\[h(\phi) = h(\phi \upharpoonright_H)+h (\bar \phi),\] where $\bar \phi=\bar\phi_{G/H}\in \End(G/H)$ is the map induced by $\phi.$ \item  The  Addition Theorem  holds in $\mathfrak X$  for endomorphisms if $\AT(G,\phi,H)$ holds for every $G\in \mathfrak X, \ \phi\in \End(G)$  and every $\phi$-invariant normal subgroup $H$ of $G.$
\item The  Addition Theorem  holds in $\mathfrak X$  for automorphisms if $\AT(G,\phi,H)$ holds for every $G\in \mathfrak X, \ \phi\in \Aut(G)$  and every $\phi$-stable normal subgroup $H$ of $G.$
\een
	\end{definition}
 Dikranjan, Goldsmith, Salce and  Zanardo \cite{DGSZ} proved that  the Addition Theorem holds   for endomorphisms  of torsion abelian groups.  This result was extended  to  all abelian groups by Dikranjan and  Giordano Bruno  \cite{DGBabelian}.  In the non-commutative case, the Addition Theorem was proved for automorphisms   of torsion groups which are either FC-groups \cite{GBST} or quasihamiltonian \cite{XST} 
 as well as for endomorphisms of torsion quasihamiltonian
FC-groups (see \cite{XST}). Extending these results, Giordano Bruno and   Salizzoni \cite{GBS} recently proved the additivity of the algebraic entropy    for endomorphisms of torsion finitely quasihamiltonian groups. Indeed, as it was shown in \cite{GBS}, the class of torsion finitely quasihamiltonian groups is a family of locally finite groups
that properly contains the class of   torsion groups that are either
FC-groups or  quasihamiltonian.

Nevertheless, the Addition Theorem  fails  for automorphisms of metabelian groups. By \cite{GBSp}, $\AT(G,id_G,H)$ does not hold, where $G=\Z_2^{(\Z)}\rtimes \Z$ is the Lamplighter group,  $H=\Z_2^{(\Z)}$  and $id_G$ is the identity automorphism of $G.$  
\subsection{Main results}
\ben \item Let  $\mathfrak X$ be a class of  solvable groups closed under taking subgroups and quotients. 
We prove that the  Addition Theorem  holds in $\mathfrak X$  for endomorphisms
if $\AT(G, \phi, G')$ holds for every $G\in\mathfrak X$ and $\phi\in \End(G).$ In case $\mathfrak X$ consists of nilpotent groups, then  the  Addition Theorem  holds in $\mathfrak X$  for automorphisms if  $\AT(G, \phi, Z(G))$ holds for every $G\in\mathfrak X$ and $\phi\in \Aut(G)$ (see Proposition \ref{prop:solvable}).

\item The reduction for solvable groups, among other things, helps us to prove  that the Addition Theorem holds for endomorphisms of two-step nilpotent torsion groups (see Theorem \ref{thm:fortwo}).
This is done by proving the inequality
$h(\phi)\geq h(\phi\upharpoonright_{G'})+h(\bar\phi)$
using Proposition \ref{prop:cen} while the converse inequality   follows from Proposition \ref{prop:mettor}. 

\item In Section \ref{sec:moreon} we provide more results concerning  locally finite groups.
It is proved in Corollary \ref{cor:simple} that the Addition Theorem holds for endomorphisms of  a  locally finite group having a  fully characteristic  finite index  simple subgroup. As a concrete example we may consider the finitary symmetric  group $\mathcal S_{fin} (\N_{+})$ (see Example \ref{ex:final}).
\een
	\subsection{Notation and terminology}
	The sets of  non-negative reals, non-negative integers and positive natural numbers are denoted by $\R_{\geq 0}, \N $  and $\N_{+}$, respectively.
	
	An element  $x$  of a group $G$ is \emph{torsion} if  the subgroup of $G$ generated by $x$,  denoted by  $\langle x\rangle$, is finite. Moreover, $G$ is \emph{torsion} if every element of $G$ is torsion.  A group is called \emph{locally finite} if every finitely generated subgroup is finite. Every locally finite group is torsion and for solvable groups  the converse also holds.
	
	A group $G$ is \emph{solvable} if  there exist $k\in \N_{+}$ and a subnormal series $$G_0=\{1\}\unlhd G_1\unlhd \cdots \unlhd G_k=G,$$  where $1$ denotes the identity element, such that the quotient group $G_j/G_{j-1}$  is abelian for every $j\in \{1,\ldots, k\}.$ In particular, $G$ is  {\em metabelian}, if $k\leq 2.$
	The subgroup $Z(G)$ denotes the \emph{center} of $G,$ and we set $Z_0(G) = \{1\}$ and  $Z_1(G) = Z(G)$.  For $n > 1$, the  \emph{$n$-th center} $Z_n(G)$ is defined as follows:
	$$Z_n(G)=\{x\in G: [x,y]\in Z_{n-1}(G)  \text{ for every } y\in G\},$$
	where $[x,y]$ denotes the commutator $xyx^{-1}y^{-1}$.  A group is {\em nilpotent} if $Z_n(G) = G$ for some $n\in \N$. In this case, its nilpotency class is the minimum of such $n$.   In particular,  $G$ is abelian or two-step nilpotent (i.e., nilpotent of class $2$) if $G'\subseteq Z(G)$, where  $G'$ is the {\em derived subgroup} of $G$, namely, the subgroup of $G$ generated by all commutators $[x,y]$, where $x,y\in G$. It is known that every nilpotent group is solvable and every nilpotent group of class at most $2$ is metabelian.
We denote by $\mathcal{P}(G)$ the power set of $G$,
while $\mathcal L(G)$ denotes the lattice of all subgroups of $G.$

We denote by $\End(G)$ the set of all endomorphisms of $G$, and	$\Aut(G)$ is its subset consisting of all automorphisms.
If $\phi\in \End(G)$, then a subgroup $H$ of $G$ is called {\em $\phi$-invariant} if $\phi(H)\subseteq H$, and $H$ is {\em $\phi$-stable} if
$\phi\in \Aut(G)$ and $\phi(H)=H.$ A subgroup $H$ of $G$ is {\em  characteristic} if $H$ is   $\phi$-invariant  for every
$ \phi\in \Aut(G),$ and $H$ is {\em  fully characteristic} if the same holds for every $ \phi\in \End(G).$

\vskip 0.4cm 
\section{The functions $\ell(-)$ and $\ell(-,-)$}
In this section we collect useful results from \cite{GBS} concerning the functions $\ell(-)$ and $\ell(-,-).$
For a group $G$ define the function $\ell:\mathcal P(G)\times \mathcal L(G)\to  \R_{\geq 0}\cup \{\infty\}$ as follows:
$\ell(X,B)=\ell(\pi(X)),$ where $\pi:G\to \{xB|\ x\in G\}$ is the canonical projection.

\begin{lemma}\cite[lemma 3.2]{GBS}\label{lem:pofl}
Let $G$ be a group, $X,X'\in \mathcal P(G)$ and $B,B'\in \mathcal L(G).$ Then:\ben[(a)]
\item the function $\ell(X,B)$ is increasing in $X$ and decreasing in $B$;
\item $\ell(XB)=\ell(X,B)+\ell(B);$
\item  $\ell(XX',B)\leq \ell(X,B)+\ell(X',B);$
\item if $BB'$ is a subgroup, $\ell(XX',BB')\leq \ell(X,B)+\ell(X',B');$
\item for $\phi\in \End(G), \ \ell(\phi(X),\phi(B))\leq \ell(X,B).$
\een
\end{lemma}
It turns out that the algebraic entropy $H(\phi, X)$ can be computed using a suitable decreasing subsequence of $\{\frac{\ell(T_n(\phi,X))}{n}\}_{n\in \N_+}.$
\begin{prop}\cite[Proposition 3.1]{GBS}\label{prop:subseq} Let $G$ be a group, $\phi\in \End(G)$ and $X\in \mathcal P_{fin}(G)$ with $1\in X.$ Then:
\ben[(a)] \item the function \[n\mapsto \frac{\ell(T_{2^n}(\phi,X))}{2^n} \] is decreasing;
\item $H(\phi, X)=\inf_{n\in \N}\frac{\ell(T_{2^n}(\phi,X))}{2^n}.$
\een
\end{prop}
The following analogous result will help us to compute the entropy of $h(\bar{\phi}).$
\begin{prop}\cite[Proposition 3.3]{GBS} \label{prop:inquo}
	Let $G$ be a group, $\phi\in \End(G), \ H$ a $\phi$-invariant normal subgroup of $G$ and $\pi:G\to G/H$ the canonical projection. Let $n\in \N$ and $X\in \mathcal P_{fin}(G)$ with $1\in X.$ Then:
	\ben[(a)] \item the function \[n\mapsto \frac{\ell(T_{2^n}(\phi,X), H)}{2^n} \] is decreasing;
	\item $H(\bar\phi_{G/H}, \pi(X))=\inf_{n\in \N}\frac{\ell(T_{2^n}(\phi,X),H)}{2^n}.$
	\een
\end{prop}
To prove   the next lemma one can use the proof of \cite[Lemma 3.4]{GBS} even though  we assume  here a weaker  condition on the finite subgroup $F$. 
\begin{lemma}\label{lem:tnsub}
Let $G$ be a group, $\phi\in \End(G), \ X\in \mathcal P_{fin}(G)$ with $1\in X.$  If $F\in \mathcal F(G)$ such that  $T_{n}(\phi,F)\in \mathcal L(G)$ for every $n\in \N_+,$ then the function 
\[n\mapsto \frac{\ell(T_{2^n}(\phi,X), T_{2^n}(\phi,F))}{2^n} \] is decreasing.
\end{lemma} 

\vskip 0.4cm  
\section{Reductions to (fully) characteristic subgroups}
 A useful property of the algebraic entropy is invariance under conjugation (see \cite[Lemma 5.1.7]{DG-islam}).
\begin{lemma}\label{lem:iuc}
	Let $G$ and $H$ be groups,  $\phi\in \End(G)$ and
	$\psi\in \End(H)$. If there exists an isomorphism $\xi : G \to H,$ then
	$h(\phi) = h(\xi \phi \xi^{-1}).$ 
\end{lemma}

In the sequel we will also use  the next property of  monotonicity for subgroups and quotients (see \cite[Lemma 5.1.6]{DG-islam}).
\begin{lemma}\label{lem:mon}
	Let G be a group, $\phi\in \End(G)$
	and H be a $\phi$-invariant  subgroup of $G$. Then \ben \item $h(\phi) \geq h(\phi \upharpoonright_H);$
	\item if $H$ is normal and $\bar \phi : G/H \to G/H$ is the endomorphism induced by $\phi$, then $h(\phi) \geq
	h(\bar\phi).$
	\een
\end{lemma}
For every group $G$ its derived subgroup $G'$  is fully characteristic, while its center $Z(G)$ is only a characteristic subgroup.
The next proposition provides partial answers to Question 5.2.11(c) and Question 5.2.12(c) of \cite{DG-islam}.
\begin{prop}\label{prop:solvable}
Let $\mathfrak X$ be a class of solvable groups closed under taking subgroups and quotients.	
	\ben \item	if $\AT(G, \phi, G')$ holds for every $G\in\mathfrak X$ and $\phi\in \End(G)$, then the Addition Theorem holds in $\mathfrak X$ for group endomorphisms;
	\item	if $\AT(G, \phi, Z(G))$ holds for every nilpotent group $G\in\mathfrak X$ and $\phi\in \Aut(G)$,  then $\AT(G, \phi, H)$ holds for every nilpotent group $G\in\mathfrak X, \phi\in \Aut(G)$ and   $H$ a $\phi$-stable normal subgroup of $G$.
		\een
\end{prop}
\begin{proof} $(1)$ Let $G\in \mathfrak X$ be a solvable group of class $n\in \N_+$. We have to prove  that $\AT(G, \phi, H)$ holds for every  $\phi\in \End(G)$ and every $\phi$-invariant normal subgroup $H$ of $G$. This will be done using induction on $n$. If $n=1$, then $G$ is abelian and the assertion follows from
	the Addition Theorem for abelian groups (see \cite[Theorem 1.1]{DGBabelian}). Using the induction hypothesis and the conjugation
	of the properties $\AT(G, \phi, G'), \AT(H, \phi\upharpoonright_{H}, H')$ and $\AT((G/H, \bar\phi, (G/H)'),$ we deduce that  $\AT(G, \phi, H)$  holds by
	\cite[Proposition 5.9]{XST}. Note that unlike here,  there $G$ is assumed to be metabelian.  Nevertheless,   $G'$ is   a solvable group of class $n-1$ for which the Addition Theorem holds by the induction hypothesis.
	
(2) Let $G\in \mathfrak X$ be a nilpotent group of class $n\in \N_+$. We have to prove  that $\AT(G, \phi, H)$ holds for every  $\phi\in \Aut(G)$ and every $\phi$-stable normal subgroup $H$ of $G$. This will be done using induction on the nilpotency class $n$. If $n=1$, then $G$ is abelian and the assertion follows from
the Addition Theorem for abelian groups.  Analogously to the proof of (1), we prove that $\AT(G, \phi, H)$ follows from the conjugation of the properties $\AT(G, \phi, Z(G)), \AT(H, \phi\upharpoonright_{H}, Z(H))$ and $\AT((G/H, \bar\phi, Z(G/H))$. 	By $\AT(G, \phi, Z(G))$, we deduce that
\begin{equation}\label{meta1}
h(\phi)=h(\phi\upharpoonright_{Z(G)})+h(\widetilde{\phi}),
\end{equation}
where $\widetilde{\phi} : G/Z(G)\to G/Z(G)$ is the  map induced by $\phi$.

As $Z(G)$ is abelian, and $Z(G)\cap H$ is a $\phi$-stable subgroup of $Z(G)$, we get by the Addition Theorem for abelian groups
\begin{equation}\label{meta2}
h(\phi\upharpoonright_{Z(G)})=h(\phi\upharpoonright_{Z(G)\cap H})+h(\widetilde{\phi\upharpoonright_{Z(G)}}),
\end{equation}
where $\widetilde{\phi\upharpoonright_{Z(G)}} \in \Aut(Z(G)/Z(G)\cap H)$ is the  map induced by $\phi\upharpoonright_{Z(G)}$.

Moreover, since $G/Z(G)$ is nilpotent of class $n-1$ , and $HZ(G)/ Z(G)$ is a $\widetilde{\phi}$-stable subgroup of $G/Z(G)$, we obtain by the induction hypothesis
\begin{equation}\label{meta3}
h(\widetilde{\phi})=h(\widetilde{\phi}\upharpoonright_{HZ(G)/Z(G)})+h(\overline{\widetilde{\phi}}),
\end{equation}
where $\overline{\widetilde{\phi}} \in \Aut((G/Z(G))/ (HZ(G)/Z(G)))$ is the  map induced by $\widetilde{\phi}$.

Hence, by Equations (\ref{meta1}), (\ref{meta2}) and (\ref{meta3}), we have
\begin{equation}\label{meta4}
h(\phi)=h(\phi\upharpoonright_{Z(G)\cap H})+h(\widetilde{\phi\upharpoonright_{Z(G)}})+h(\widetilde{\phi}\upharpoonright_{HZ(G)/Z(G)})+h(\overline{\widetilde{\phi}}).
\end{equation}

\noindent
\textbf{Claim 1} $h(\phi\upharpoonright_H)=h(\phi\upharpoonright_{Z(G)\cap H})+h(\widetilde{\phi}\upharpoonright_{HZ(G)/Z(G)})$.
\begin{proof}
	
	By $\AT(H, \phi\upharpoonright_{H}, Z(H))$, we deduce that
	\begin{equation}\label{meta6}
	h(\phi\upharpoonright_{H})=h(\phi\upharpoonright_{Z(H)})+h(\widetilde{\phi\upharpoonright_{H}}),
	\end{equation}
	where $\widetilde{\phi\upharpoonright_{H}} : H/ Z(H)\to H/ Z(H)$ is the  map induced by $\phi\upharpoonright_{H}$.
	
	As $Z(H)$ is abelian, and $H\cap Z(G)$ is a $\phi$-stable subgroup of $Z(H)$, we obtain by the Addition Theorem for abelian groups
	\begin{equation}\label{meta5}
	h(\phi\upharpoonright_ {Z(H)})=h(\phi\upharpoonright_{H\cap Z(G)})+h(\xi),
	\end{equation}
	where $\xi : Z(H)/(H\cap Z(G))\to Z(H)/(H\cap Z(G))$ is the  map induced by $\phi\upharpoonright_{Z(H)}$. Hence, to prove Claim 1, it suffices to show that
	\begin{equation}\label{meta8}
h(\widetilde{\phi}\upharpoonright_{HZ(G)/Z(G)}) = h(\xi)+	h(\widetilde{\phi\upharpoonright_{H}}).
	\end{equation}
	Let $\psi\in \Aut(H/(H\cap Z(G)))$ be the map induced by $\phi\upharpoonright_{H}.$  Then $H/(H\cap Z(G))$ is nilpotent of class less than $n$ having
	 $Z(H)/(H\cap Z(G))$ as a $\psi$-stable subgroup. By the induction hypothesis, $$\AT(H/(H\cap Z(G)), \psi, Z(H)/(H\cap Z(G)))$$ holds.
	Moreover, $\xi= \psi\upharpoonright_{Z(H)/(H\cap Z(G))}$ and the automorphisms $\psi, \bar{\psi}$ are conjugated, respectively, to $$\phi\upharpoonright_ {Z(H)}, \widetilde{\phi\upharpoonright_{H}},$$   where $\bar{\psi}\in \Aut (H/(H\cap Z(G)))/(Z(H)/(H\cap Z(G)))$ is the map induced by  $\psi.$

	Therefore,   Equation (\ref{meta8})  follows from $\AT(H/(H\cap Z(G)), \psi, Z(H)/(H\cap Z(G)))$ and Lemma \ref{lem:iuc}.
\end{proof}

\noindent
\textbf{Claim 2} $h(\bar\phi)=h(\widetilde{\phi\upharpoonright_{Z(G)}})+h(\overline{\widetilde{\phi}})$.
\begin{proof}
By $\AT((G/H, \bar\phi, Z(G/H))$ we have, 	\begin{equation}\label{meta9}
h(\bar\phi)=h(\bar\phi\upharpoonright_{Z(G/H)})+h(\delta),
\end{equation} 	where $\delta\in \Aut((G/H)/Z(G/H))$ is the  map induced by $\bar\phi$. As $Z(G)H/H$ is a $\bar\phi$-stable  subgroup of the abelian group $Z(G/H)$, we have
\begin{equation}\label{meta10}
h(\bar\phi\upharpoonright_{Z(G/H)})=h(\bar\phi\upharpoonright_{Z(G)H/H})+h(\varphi),
	\end{equation}
		where $\varphi\in \Aut(Z(G/H)/(Z(G)H/H))$ is the  map induced by $\bar\phi\upharpoonright_{Z(G/H)}$.
 The map $\widetilde{\phi\upharpoonright_{Z(G)}}$ is conjugated to $\bar\phi\upharpoonright_{Z(G)H/H}$. This fact implies that $h(\widetilde{\phi\upharpoonright_{Z(G)}})=h(\bar\phi\upharpoonright_{Z(G)H/H})$	. By (\ref{meta9}) and (\ref{meta10}) we get 
 \begin{equation}\label{meta11}
 h(\bar\phi)=h(\widetilde{\phi\upharpoonright_{Z(G)}})+h(\varphi)+h(\delta).
 \end{equation}
 So, to prove Claim 2, it suffices to show that  \begin{equation}\label{meta12}h(\overline{\widetilde{\phi}})= h(\varphi)+h(\delta).\end{equation}
Let $M=(G/H)/(HZ(G)/H)$, $N=Z(G/H)/(Z(G)H/H)$  and $\eta\in \Aut(M)$ be the map induced by $\bar\phi$. As $M$ is nilpotent of class less than $n$, and $N$ is an $\eta$-stable subgroup of $M$, we deduce that $\AT(M, \eta, N)$ holds by the induction hypothesis.
We use Invariance under conjugation to conclude that  (\ref{meta12}) is satisfied.
\end{proof}
	Claim 1, Claim 2 and Equation (\ref{meta4}) complete the proof of (2), i.e.,\
\[
h(\phi)=h(\phi\upharpoonright_H)+h(\bar\phi).
\qedhere\]
\end{proof}

\vskip 0.4cm 
\section{When $G$ is a two-step nilpotent torsion group}
We prove in Theorem \ref{thm:fortwo} below that the Addition Theorem holds for endomorphisms of two-step nilpotent torsion groups.
\begin{prop}\label{prop:cen}
	Let $G$ be a group, $\phi\in \End(G)$ and $N$ be a $\phi$-invariant central torsion subgroup of $G.$ Then, 
	\[
	h(\phi)\geq h(\phi\upharpoonright_{N})+h(\bar\phi),
	\]
	where $\bar\phi\in \End(G/N)$ is the map induced by $\phi.$
\end{prop}
\begin{proof}
Let $E$ be a finite subgroup of $N$ and $C=\pi(B)$, where $B$ is a finite subset of $G$ and $\pi: G\to G/N$ is the quotient homomorphism.
As $E$ is a finite subgroup of the $\phi$-invariant central subgroup $N$, we deduce that $\phi^k(E)$ is a central subgroup of  $G$ for every $k\in \N.$ In particular, $T_n(E)$ is central and $T_n(\phi,EB)=T_n(\phi,E)\cdot T_n(\phi,B)$ for every $n\in \N_{+}.$ This implies that $\ell(T_n(\phi,EB))=\ell(T_n(\phi,E))+\ell(T_n(\phi,B),T_n(\phi,E)).$  By Lemma \ref{lem:pofl}(a),
\[\ell(T_n(\phi,E))+\ell(\pi(T_n(\phi,B)))=\ell(T_n(\phi,E))+ \ell(T_n(\phi,B),N)\leq\]\[\leq \ell(T_n(\phi,E))+ \ell(T_n(\phi,B),T_n(\phi,E))=\ell(T_n(\phi,EB)).\]
As $\pi(T_n(\phi,B))=T_n(\bar{\phi},C)$, it follows that $H(\phi\upharpoonright_{N},E)+H(\bar{\phi}, C)\leq H(\phi, EB)\leq h(\phi)$.
Since $E,B$ were chosen arbitrarily, we deduce that 	\[
h(\phi)\geq h(\phi\upharpoonright_{N})+h(\bar\phi).\qedhere
\]
\end{proof}
	\begin{remark}\label{rem:coin} Another algebraic  entropy is defined	in 	\cite[Remark 5.1.2]{DG-islam} as follows. The $n$-th $\phi$-trajectory of $X$ is
	\[T^{\#}_n(\phi,X)=\phi^{n-1}(X)\cdot \ldots \cdot \phi(X)\cdot X.\] The algebraic entropy of $\phi$ with respect
	to $X$ is $H^{\#}(\phi, X)=\lim_{n\to \infty}\frac{\ell(T^{\#}_n(\phi,X))}{n}$ and the  algebraic entropy of $\phi$ is
	$h^{\#}(\phi)=\sup\{H^{\#}(\phi, X)|\ X\in \mathcal C\},$ where   $\mathcal C$ is any cofinal subfamily of  $\mathcal P_{fin}(G)$. Let us see that $h^{\#}$  coincides with $h.$ To this aim, let $X\in \mathcal P_{fin}(G)$. Then  $\widetilde{X}=X\cup X^{-1}$ is a  symmetric set containing $X$ and 
	$(T_n(\phi,\widetilde{X}))^{-1}=T^{\#}_n(\phi,\widetilde{X})$. This implies that $H^{\#}(\phi, \widetilde{X})=H(\phi, \widetilde{X})$ and 
	$h^{\#}(\phi)=	h(\phi)$, as $\mathcal C=\{\widetilde{X}| \ X\in \mathcal P_{fin}(G) \}$ is cofinal in $\mathcal P_{fin}(G)$.
\end{remark}
\begin{lemma}\label{lem:cfsub}
Let  $G$ be a group, $\phi\in \End(G)$ and $F$ be a subgroup of  $G$. Then:
\ben \item $\langle T_n(\phi,F)  \rangle=\bigcup_{m\in \N_+}(T_n(\phi,F)\cup T_n(\phi,F)^{-1})^m.$ In particular, if $G$ is locally finite and $F$ is finite, then $$\langle T_n(\phi,F)  \rangle=(T_n(\phi,F)\cup T_n(\phi,F)^{-1})^m$$ for some $m\in \N_+;$

\item $\langle T_n(\phi,F)  \rangle =T_n(\phi,F) \cdot E_n,$ where $E_n=\langle T_n(\phi,F) \rangle \cap G',$  for every $n\in \N_{+}.$ 
\een
\end{lemma}
\begin{proof}
(1) Since the set $X_n=T_n(\phi,F)\cup T_n(\phi,F)^{-1}$ is symmetric and contains $1,$  it generates a subgroup of the form $\bigcup_{m\in \N_+}X_n^m.$ 
When $G$ is locally finite and $F$ is finite we may use the finiteness of $\langle T_n(\phi,F)  \rangle$ and the containment $X_n^m\subseteq X_n^{m+1}$ to prove the last assertion.\\
(2) In the notation of   Remark \ref{rem:coin},
\[T_n(\phi,F)^{-1}=T^{\#}_n(\phi,F)=\phi^{n-1}(F)\cdot \ldots \cdot \phi(F)\cdot F \] as $F$ is a subgroup of $G.$ Let  $g_1,\ldots, g_n\in G$, where $n\geq 2.$ Then \[\pi(g_1\cdot g_2\cdots g_{n-1}\cdot g_n)=\pi(g_n\cdot g_{n-1}\cdots g_2\cdot g_1),\] where $\pi:G\to G/G'$ is the quotient map, since  $G/G'$ is abelian.  It follows that \[\pi(T_n(\phi,F))=\pi(T^{\#}_n(\phi,F))=\pi (T_n(\phi,F)^{-1}).\] Moreover,  as $F$ is a subgroup of $G$ and $G/G'$ is abelian we get that \[\pi(T_n(\phi,F)\cup T_n(\phi,F)^{-1})^m=\pi(F^m\cdot \phi(F^m)\ldots \cdot \phi^{n-1}(F^m))=\pi(F\cdot \phi(F)\ldots \cdot \phi^{n-1}(F))=\pi(T_n(\phi,F)),\] for every $m\in \N_+.$ 
Using also (1) we deduce that $\pi(T_n(\phi,F))=\pi(\langle T_n(\phi,F)  \rangle).$
This implies that  \[\langle T_n(\phi,F)  \rangle = T_n(\phi,F)\cdot E_n,\] where $E_n=\langle T_n(\phi,F) \rangle \cap G'.$ 
\end{proof}
\begin{prop}\label{prop:mettor}
	Let $G$ be a  torsion metabelian group, $\phi\in \End(G)$. Then,
	$h(\phi)\leq h(\phi\upharpoonright_{G'})+h(\bar\phi).$
\end{prop}
\begin{proof}
	Since $G$ is a torsion metabelian  group it is locally finite by \cite[Proposition 1.1.5]{DX}. Recall that in this case
	$h(\phi)=\sup\{H(\phi, F))|\ F\in \mathcal F(G)\}.$ 
Let $D$ be a finite subgroup of $G$ and  let $C=\pi(D),$ where $\pi: G\to G/G'$ is the quotient homomorphism. Fix $\varepsilon>0.$ By Proposition \ref{prop:inquo}, there exists $M\in \N$ such that, for every $n\geq M$,
\begin{equation}\label{eq:last1}
\frac{\ell(T_{2^n}(\phi,D),G')}{2^n}\leq H(\bar{\phi}, C)+\varepsilon.
\end{equation}
 For $T=T_{2^M}(\phi,D)$ and $E=\langle T \rangle \cap G'$ we have $\langle T\rangle =TE$ and $\langle T\rangle G'=TEG'=TG',$  by Lemma  \ref{lem:cfsub}. Note that $E$ and  $S=T_{2^M}(\phi,E)$ are  finite abelian as $G$ is 
 torsion metabelian.  On the one hand,  \[\ell(T,E)=\log[\langle T\rangle:E]= \log[\langle T\rangle:\langle T \rangle \cap G']= \log [\langle T\rangle G':G']=\log [T G':G']=\ell(T,G').\]    On the other hand,  \[\ell(T,G')\leq \ell(T,S)\leq \ell(T,E) \] by Lemma \ref{lem:pofl}(a). Hence,
\begin{equation} \label{eq:last2}
\ell(T,G')=\ell(T,S).
\end{equation}
Let $n\geq M.$ By Lemma \ref{lem:tnsub} and equations (\ref{eq:last1}) and (\ref{eq:last2}),
\begin{equation} \label{eq:last3}
\frac{\ell(T_{2^n}(\phi,D),T_{2^n}(\phi,E))}{2^n}\leq \frac{\ell(T,S)}{2^M}= \frac{\ell(T,G')}{2^M}\leq H(\bar{\phi}, C)+\varepsilon\leq  h(\bar{\phi})+\varepsilon.
\end{equation}
By Proposition \ref{prop:subseq}, there exists $N\geq M,$ such that for every $n\geq N,$
\begin{equation} \label{eq:last4}\frac{\ell(T_{2^n}(\phi,E))}{2^n}\leq H(\phi\upharpoonright_{G'},E)+\varepsilon\leq h(\phi\upharpoonright_{G'})+\varepsilon,
\end{equation}
and also
\begin{equation} \label{eq:last5}
H(\phi,D)\leq\frac{\ell(T_{2^n}(\phi,D))}{2^n}.
\end{equation}
By Lemma \ref{lem:pofl}(b),
\begin{equation} \label{eq:last6}
\ell(T_{2^n}(\phi,D))\leq \ell(T_{2^n}(\phi,D) T_{2^n}(\phi,E))=\ell(T_{2^n}(\phi,D), T_{2^n}(\phi,E))+\ell(T_{2^n}(\phi,E)).
\end{equation}
It follows from (\ref{eq:last3}), (\ref{eq:last4}), (\ref{eq:last5})and (\ref{eq:last6})  that
\[ H(\phi,D)\leq\frac{\ell(T_{2^n}(\phi,D))}{2^n}\leq \frac{\ell(T_{2^n}(\phi,D), T_{2^n}(\phi,E))}{2^n}+\frac{\ell(T_{2^n}(\phi,E))}{2^n}\leq  h(\bar{\phi})+h(\phi\upharpoonright_{G'})+2\varepsilon.\]
This completes the proof as the latter holds for any finite subgroup $D$ and any $\varepsilon>0.$
\end{proof}	
\begin{remark}
 It is worth noting that Proposition \ref{prop:mettor}  is no longer true in case the metabelian group $G$ is not torsion. Indeed, consider the Lamplighter group 
 $G=\Z_2^{(\Z)}\rtimes \Z$ which was mentioned in the introduction.
Using the arguments appearing in
\cite[Example 2.7]{GBS}, one can show that for $\phi=id_G$ it holds that
$h(\phi\upharpoonright_{G'})=h(\bar\phi)=0$ while $h(\phi)=\infty.$ It follows that
\[h(\phi)> h(\phi\upharpoonright_{G'})+h(\bar\phi).\]

\end{remark}
\begin{thm}\label{thm:fortwo}
If $G$ is a   two-step nilpotent torsion group, $\phi\in\End(G)$ and $H$ is a $\phi$-invariant normal subgroup of $G,$ then $\AT(G,\phi,H)$ holds.
\end{thm}
\begin{proof}
By Proposition \ref{prop:solvable}(1), it suffices to prove that $\AT(G,\phi,G')$ holds. Taking into account that $G'\subseteq Z(G)$ and using  Proposition \ref{prop:cen} we have,
	\[
h(\phi)\geq h(\phi\upharpoonright_{G'})+h(\bar\phi).
\]
By Proposition \ref{prop:mettor},  the converse inequality also holds and we complete the proof.
\end{proof}

\vskip 0.4cm  
\section{More on locally finite groups}\label{sec:moreon}
Let $G$ be a group and $\phi\in\End(G).$ In the sequel we say that $\AT(G,\phi)$ holds if $\AT(G,\phi, H)$ holds for every
$\phi$-invariant normal subgroup $H.$  If this happens for   any $\phi\in\End(G),$ then we say that $\AT(G)$ holds.
\begin{prop}\label{prop:nowsimple}
	Let $G$ be a locally finite group and $\phi\in \End(G).$ Then the family $$\mathcal{L}(G,\phi)=\{H \ \text{is a } \ \phi\text{-invariant  normal subgroup
	  of} \ G \ \text{and} \ \AT(H,\phi\upharpoonright_{H})\ \text{holds} \}$$ contains  a maximal element (with respect to inclusion).
\end{prop}
\begin{proof}
	$\mathcal{L}(G,\phi)$ is not empty as it contains $\{1\}.$	 Let $\{H_i\}_{i\in I}$ be a chain in $\mathcal{L}$. By Zorn's Lemma, it suffices to show  that
	$H=\bigcup_{i\in I} H_i\in  \mathcal{L}(G,\phi).$ Clearly, if each $H_i$ is a  $\phi$-invariant normal subgroup  of $G$, then so is $H$. Let us see that $\AT(H,\phi\upharpoonright_{H},N)$ holds, where $N$ is a  $\phi$-invariant normal subgroup
	of $H$. Since $H$ is the direct limit of the $\phi$-invariant subgroups $\{H_i| \ i\in I\}$ it follows from \cite[Proposition 5.1.10]{DG-islam} that \[h(\phi\upharpoonright_{H})=\sup_{i\in I}h(\phi\upharpoonright_{H_i}).\]
	
Similarly, we have $h(\phi\upharpoonright_{ N})=\sup_{i\in I}h(\phi\upharpoonright_{H_i\cap N})$ and  $$h(\bar\phi_{H/N})=\sup_{i\in I}h(\bar\phi_{H_iN/ N})=\sup_{i\in I}h(\bar\phi_{H_i/(H_i\cap N)}),$$ where the last equality follows from Lemma \ref{lem:iuc}.
	By our assumption, $\AT(H,\phi\upharpoonright_{H_i}, N\cap H_i)$ holds for every $i\in I.$ Using this and the previous equalities we obtain
	\[h(\phi\upharpoonright_{H})=\sup_{i\in I}h(\phi\upharpoonright_{H_i})=\sup_{i\in I}h(\phi\upharpoonright_{H_i\cap N})+\sup_{i\in I}h(\bar\phi_{H_i/(H_i\cap N)})=h(\phi\upharpoonright_{N})+h(\bar\phi_{H/N}), \] as needed.
\end{proof}
\begin{prop}\label{prop:zero}
	Let $G$ be a locally finite group, $\phi \in \End(G)$ and  $H$ be a   $\phi$-invariant normal subgroup of $G.$ \ben \item If $h (\bar \phi)=0$, then $h(\phi) =h(\phi \upharpoonright_H);$
	\item if, in addition,  $\AT(H,\phi\upharpoonright_{H})$ holds, then $\AT(G,\phi)$ holds. \een
\end{prop}
\begin{proof}
(1) By Lemma \ref{lem:mon}(1),	it suffices to show that $h( \phi)\leq h(\phi\upharpoonright_H)$,  i.e., that $H(\phi, E)\leq h(\phi\upharpoonright_H)$ for an arbitrary finite 
	subgroup $E$ of $G.$ Let $\pi : G\to G/H$ be the canonical homomorphism and let $E_1=\pi(E)$. Since $\bar\phi$ has zero entropy and $G/H$ is locally finite, it follows from \cite[Proposition 4.5]{GBSp} that there exists $m\in \N_+$ such that 
	 $T_m(\bar\phi, E_1)$ is $\bar\phi$-invariant. This implies that $\phi ^{m}(E) \subseteq  T_m(\phi, E)\cdot H$. As $E$ is finite there exists a finite subset $F$ of $H$ such that $\phi ^{m}(E)\subseteq T_m(\phi, E)\cdot F$. Since $H$ is locally finite, there exists a finite subgroup $E_2$ of $H$ containing $F$, so we have
	\begin{equation}\label{eq:one}
	\phi ( T_m(\phi, E) )\subseteq T_m(\phi, E)\cdot E_2.
	\end{equation}

	We now prove that
	\begin{equation}\label{eq:indu}
	T_{m+k}(\phi, E)\subseteq T_m(\phi, E)T_k(\phi, E_2)
	\end{equation}
	using induction on $k$. For $k=1$, $T_{m+1}(\phi, E)\subseteq T_m(\phi, E)T_1(\phi, E_2)$ follows from Equation (\ref{eq:one}). Assuming that $T_{m+k}(\phi, E)\subseteq T_m(\phi, E)T_k(\phi, E_2)$, one obtains
	\begin{equation*}\begin{split}
	T_{m+k+1}(\phi, E)&=E\phi(T_{m+k}(\phi, E))\subseteq
	\\ &\subseteq E\phi(T_m(\phi, E))\phi(T_k(\phi, E_2))=
	\\ &=T_{m+1}(\phi, E)\phi(T_k(\phi, E_2))\subseteq
	\\ &\subseteq T_m(\phi, E)E_2\phi(T_k(\phi, E_2))=
	\\ &= T_m(\phi, E)T_{k+1}(\phi, E_2).
	\end{split}\end{equation*}
	Thus Equation (\ref{eq:indu}) holds, and we have
	\[
	\frac{\log|T_{m+k}(\phi, E)|}{m+k} \cdot \frac{m+k}{k}\leq \frac {\log|T_m(\phi, E)|}{k}+ \frac{\log|T_k(\phi, E_2)|}{k}.
	\]
	
	Since $m$ is fixed, letting $k\to \infty$ (so $k+m\to \infty$ as well), we deduce that
	\[
	H(\phi, E)\leq H(\phi, E_2)\leq h(\phi\upharpoonright_H). 
	\]
	(2) Let $M$ be a  $\phi$-invariant normal subgroup of $G.$ By our assumption, $\AT(H,\phi\upharpoonright_{H}, M\cap H)$ holds as $M\cap H$
	is $\phi$-invariant normal subgroup of $H.$ This means that  \begin{equation}\label{ex:split1}
	h(\phi\upharpoonright_{M\cap H})+ h(\bar\phi_{H/H\cap M})=h(\phi\upharpoonright_H),  
	\end{equation}
	where $\phi_{H/H\cap M}$ is the endomorphism induced by $\phi\upharpoonright_H.$   Since $HM/H$ is a subgroup of $G/H$ and $h (\bar \phi_{G/H})=0$ it follows from Lemma \ref{lem:mon}(1) that   $h (\bar \phi_{HM/H})=0.$ By Lemma \ref{lem:iuc}, $\bar\phi_{M/M\cap H}$ has zero entropy as this endomorphism is conjugated to $\bar\phi_{HM/ H}.$ Applying item (1) to the locally finite group $M$ and the endomorphism $\phi\upharpoonright_M$ we deduce that
	\begin{equation}\label{ex:split2}
	h(\phi\upharpoonright_M)=	h(\phi\upharpoonright_{M\cap H}).
	\end{equation}
  As $G/HM\cong (G/M)/(HM/M)$ is a quotient of $G/H$ and
$\bar\phi_{G/ H}=0,$ we deduce by Lemma \ref{lem:iuc} and Lemma \ref{lem:mon}(2) that $\tilde \phi\in  \End((G/M)/(HM/M))$ has zero entropy, 
where $\tilde \phi$ is the map induced by $\bar\phi_{G/ M}.$ By  Lemma \ref{lem:iuc}, $h(\bar\phi_{H/H\cap M})=h(\bar\phi_{HM/ M}).$
Applying item (1) to the locally finite group $G/M$ and the endomorphism $\bar\phi_{G/M}$ we obtain
	\begin{equation}\label{ex:split3}
h(\bar\phi_{G/ M})=h(\bar\phi_{HM/ M})=h(\bar\phi_{H/H\cap M}).
\end{equation}
From (1) we have \begin{equation}\label{ex:split4} h(\phi) =h(\phi \upharpoonright_H).
\end{equation}
It follows from (\ref{ex:split1}), (\ref{ex:split2}), (\ref{ex:split3}) and (\ref{ex:split4}) that $\AT(G,\phi, M)$ holds.
\end{proof}
\begin{corol}\label{cor:simple}
If $G$ is a locally finite group having a  fully characteristic  finite index simple subgroup $H$, then  $\AT(G)$ holds.
\end{corol}
\begin{proof}
Let $\phi\in \End(G).$ Since $H$ is fully characteristic 	it is $\phi$-invariant normal subgroup of $G$. Since $[G:H]< \infty$ it follows that
 $h (\bar \phi)=0.$ Clearly, $\AT(H)$ holds as $H$ is simple. By Proposition \ref{prop:zero}, $\AT(G)$ holds.
\end{proof}
\begin{example}\label{ex:final}
Let us prove that $\AT(G)$ holds, where $G=\mathcal S_{fin} (\N_{+})$ is the  finitary symmetric group  which consists of all permutations on $\N_{+}$ of finite support. Note  that $G$ is a locally finite group that is not finitely quasihamiltonian (see \cite[Example 2.1(e)]{GBS}).  It is known that   $\Alt(\N_{+}),$ the infinite group of all even permutations, is a fully characteristic simple subgroup of $G$ of index $2$.  By Corollary \ref{cor:simple},
$\AT(G)$ holds.
\end{example}

\vskip 0.1cm  
\noindent \textbf{Acknowledgments.}  It is a pleasure to thank the referee for the  useful suggestions. In particular, for simplifying the proof of Proposition \ref{prop:nowsimple} using \cite[Proposition 5.1.10]{DG-islam}.

\end{document}